\documentclass[a4paper,smallextended,envcountsame]{svjour3}
\usepackage[pdftex]{graphicx}

\usepackage[utf8]{inputenc}

\usepackage[round]{natbib}

\usepackage{amsmath,amssymb}

\journalname{Algorithmica}
\date{\today}

\newcommand\reals{\mathbb{R}}
\newcommand\naturals{\mathbb{N}}
\newcommand\integers{\mathbb{Z}}

\newcommand\bigo{\mathcal O}
\newcommand\abs[1]{\lvert#1\rvert}
\DeclareMathOperator\erf{erf}
\let\epsilon\varepsilon

\newcommand\dd{\mathrm d}

\newcommand{\be}{\begin{equation}}
\newcommand{\ee}{\end{equation}}
\newcommand{\ef}[1]{\, #1}

\renewcommand{\a}{\alpha}

\newcommand\rd{{\mathrm d}}

\newcommand{\reof}{\mathfrak{Re}}

\begin{document}

\author{Axel Bacher\and Andrea Sportiello}

\title
{Complexity of anticipated rejection algorithms\\
and the Darling--Mandelbrot distribution}
\titlerunning{Complexity of anticipated rejection and the Darling--Mandelbrot distribution}

\institute{Axel Bacher$^{\dagger}$, RISC, Johannes Kepler University, 
Altenbergerstraße 69, A-4040
Linz,
  \hbox{Austria.}
\quad
\email{abacher@risc.uni-linz.ac.at}\\
Andrea Sportiello$^{\ddagger}$, LIPN, University Paris Nord, 
99 av.~J.-B.\ Clément, 93430
Villetaneuse, France.
\quad
\email{andrea.sportiello@lipn.univ-paris13.fr}
\\
${\dagger}$\ 
{Supported by FWF project F050-04 (Austria).}
\\
${\ddagger}$\ 
{Supported by ANR Magnum project BLANC 0204 (France).}
}


\maketitle

\begin{abstract}
We study in limit law the complexity of some anticipated rejection random
sampling algorithms. We express this complexity in terms of a probabilistic
process, the threshold sum process. We show that, under the right conditions,
the complexity is linear and admits as a limit law a so-called
Darling--Mandelbrot distribution, studied by Darling (1952) and Lew (1994). We
also give an explicit form to the density of the Darling--Mandelbrot
distribution and derive some of its analytic properties.

\end{abstract}

\section{Introduction}
\label{sec.Intro}

This paper aims at answering the following algorithmic question: consider a
program~$P$ that performs a random number of elementary operations and then
terminates. Our goal is to have~$P$ performing $n$ operations in one run. To
do that, we run the program~$P$ until it reaches~$n$ operations; if it
terminates before that, we simply restart it. The question is, how many
elementary operations must we perform to reach this goal?

Algorithms of this type are abundant in the field of \emph{random sampling},
where they are known as \emph{anticipated rejection} algorithms. Given a class
of discrete objects, a random sampling algorithm takes an integer~$n$ as input
and outputs a random object of size~$n$ according to a specific (usually
uniform) distribution. Given a random sampling algorithm for a class~$\mathcal
A$ and a subclass~$\mathcal B$ of~$\mathcal A$, an element of~$\mathcal B$ can
be sampled using a \emph{rejection} algorithm: we repeatedly sample elements
of~$\mathcal A$ until we find one in~$\mathcal B$. This algorithm can be
improved when it is possible to know in advance, during the sampling
procedure, that the drawn element is not going to be in~$\mathcal B$: we can
then prematurely reject the sample and start over, saving computing time. This
scheme is called \emph{anticipated rejection}. Assuming that sampling an
element of~$\mathcal A$ costs~$n$ elementary operations, this fits into the
framework outlined above.

Such algorithms are found for example in \cite[]{barcucci,barcucci2},
sampling prefixes of Motzkin paths (the so-called \emph{Florentine
algorithm}). Somewhat miraculously, this algorithm achieves an average linear
time complexity, as, on average, the number of necessary trials is $\mathcal
O(\sqrt n)$ and each trial costs $\mathcal O(\sqrt n)$. We show that this
phenomenon is not isolated, but rather happens in a wider range of cases.
Other algorithms of this family exist, sampling Schröder prefixes
\cite[]{penaud}, unary-binary trees \cite[]{motzkin} and constrained random
walks.

In this paper, we study the full limit distribution of the complexity of these
algorithms. This problem leads us to define a probabilistic process, the
\emph{threshold sum process}. Our main result is that, if the base
distribution has a tail with exponent~$\alpha$ in a certain range, this
process admits a limit distribution depending only on~$\alpha$. This
\emph{universality} phenomenon is reminiscent of Lévy's well-known theory of
$\alpha$-stable distributions, which also deals with sums of independent
random variables \cite[]{gnedenko}.

Surprisingly, our limiting distribution has already been studied in relation
to a different problem, namely, the ratio between the sum and the maximum
of a fixed number of i.i.d.\ random variables. It was first studied by Darling
\cite[]{darling}, then apparently by Mandelbrot in unpublished work, and by
Lew \cite[]{lew}, who named it the \emph{Darling--Mandelbrot distribution}.
This distribution has a parameter $\alpha$, with $0<\alpha<1$; it is supported
on $\reals_+$ and is defined by its characteristic function:
\begin{equation}
\phi_\alpha(s) = \frac{(-is)^{-\alpha}}{-\alpha\gamma(-\alpha,-is)}
=
\Biggl( 1-\sum_{n=1}^\infty\frac\alpha{n-\alpha}
\frac{(is)^n}{n!} \Biggr)^{-1}
\text,
\label{dm}
\end{equation}
where in the first expression, $\gamma(\cdot,\cdot)$ denotes the lower
incomplete gamma function\footnote{
Given the
  definition of the Gamma function $\Gamma(y) = \int_0^{\infty}
  x^{y-1} e^{-x} \rd x$, the upper and lower incomplete versions
are defined through the
  corresponding integrals on modified domains $\Gamma(y,z) =
  \int_z^{\infty} \cdot\;$ and $\gamma(y,z) = \Gamma(y)-\Gamma(y,z) =
  \int_0^z \cdot\;$, respectively. Non-positive real values of~$z$ are reached
by analytic continuation.}.
The second expression allows to easily extract the
moments of the distribution as rational functions of $\a$. Lew showed that the
distribution has an exponential tail; moreover, we show that its density is
non-analytic at all integer points. Both properties contrast with the Lévy
distributions, which have an analytic density and a heavy tail.

In the case of the Florentine algorithm (which corresponds to an
exponent~$\alpha=1/2$, as seen below), an expression of the Laplace transform
of the limit distribution already appears in \cite[]{louchard}, namely:
\[\frac1{e^{-z} + \sqrt{\pi z}\erf(\sqrt z)}\text.\]
We readily check that this expression is equivalent to $\phi_{1/2}(iz)$, the
Laplace transform of the Darling-Mandelbrot distribution of parameter $1/2$.

The paper is organized as follows. In Section~\ref{sec:TSP}, we define the
threshold sum process and show that, under some conditions, its limit
distribution is a Darling--Mandelbrot distribution. In
Section~\ref{sec:density}, we give an explicit form for the
Darling--Mandelbrot density and give analytic results expanding those of Lew.
Finally, in Section~\ref{sec:applications}, we use these results to analyse some
anticipated rejection algorithms. 

\section{The threshold sum process} 
\label{sec:TSP}
\label{sec.TSP}

In the following, let $(X_i)_{i\ge0}$ be a sequence of independent and
identically distributed random variables with values in $\naturals$ or
$\reals_+$ and unbounded support. We denote by $F(x)$ the complementary
cumulative distribution function of the $X_i$'s:
\[F(x) = \mathbb P(X_i\ge x)\text.\]
Let $t\ge0$ and let $I(t)$ be the smallest index such that $X_{I(t)}\ge t$.
Define the \emph{threshold sum process} (TSP) $Y_t$ as:
\begin{equation*} 
Y_t = X_0 + \dotsb + X_{I(t)-1}\text.
\end{equation*}
The number $t$ is called the \emph{threshold}. This process resembles the
classical sum of independent random variables, but the number of
summands~$I(t)$ is here a random variable depending on the real parameter~$t$.
Our main result on this process is the following.

\begin{theorem}
\label{thm:main}
Assume that, as $x$ tends to infinity, $F(x)$ is equivalent to $c\,x^{-\alpha}$
for some $c>0$ and $\alpha>0$. Then, as $t$ tends to infinity, the random
variable $Y_t$ satisfies:
\begin{itemize}
\item if $\alpha<1$, then $Y_t/t$ converges in distribution to the
Darling--Mandelbrot law of parameter~$\alpha$;
\item if $\alpha=1$, then $Y_t/(t\log t)$ converges in distribution to the
exponential law;
\item if $\alpha>1$, then $Y_t/(t^\alpha c^{-1}\mu)$, where $\mu = \mathbb
E(X_i)$, converges in distribution to the exponential law.
\end{itemize}
\end{theorem}

To us, the most interesting case is $\alpha < 1$, where the behavior of~$Y_t$
is strongly universal in that it only depends on the exponent~$\alpha$.
Moreover, the scaling factor is always~$t$ in that range (this is different
from Lévy's theory of sums of i.i.d.\ random variables, where the scaling
factor is a power of~$t$ depending on~$\alpha$). For $\alpha=1$, the
scaling factor is augmented by a~$\log t$ factor; for $\alpha>1$, the scaling
factor is higher and we have a lesser form of universality, with the limit
scaled by~$\mu/c$. Consequences of these facts to the analysis of algorithms
are discussed in Section~\ref{sec:applications}.



\begin{proof}
We prove this result using Lévy's Continuity Theorem, which states that a
sequence of random variables tends in distribution to some limit if their
characteristic functions converge pointwise to the characteristic function of
the limit distribution.

Let $\psi_t(s) = \mathbb E(e^{isY_t/\tau})$ be the characteristic function of
the random variable $Y_t/\tau$, where $\tau$ is a scaling factor (depending
on~$t$) to be specified later on. The index $I(t)$ is geometrically
distributed with parameter~$F(t)$, which is the probability that $X_i\ge t$.
The random variables $X_0,\dotsc,X_{I(t)-1}$ are constrained to be less
than~$t$; let $\chi_t(s) = \mathbb E(e^{isX/t}|X < t)$ be the characteristic
function of such a constrained variable. We have:
\[\psi_t(s) = \frac{F(t)}{1 - \bigl(1-F(t)\bigr)\chi_t(s/\tau)}\text.\]
On the other hand, we have:
\[\chi_t(s) = \frac1{1 - F(t)}\sum_{n=0}^\infty
M_{t,n}\frac{(is)^n}{n!}\text,\qquad\text{with}\qquad M_{t,n} = \int_0^tx^ndF(x)\text.\]
We therefore have:
\[\psi_t(s) = \frac{F(t)}
{\displaystyle1 - \sum_{n=0}^\infty\frac{M_{t,n}}{\tau^n}\frac{(is)^n}{n!}}
= \frac1{\displaystyle1 - \sum_{n=1}^\infty \frac{M_{t,n}}{\tau^n F(t)}
\frac{(is)^n}{n!}}\text,\]
where the last simplification follows from the fact that $M_{t,0} = 1 - F(t)$.
%

%
%
Consider first the case where $\alpha < 1$. Using integration by parts, we
find that the term $M_{t,n}$ satisfies as $t$ tends to infinity:
\[M_{t,n} = -t^nF(t) + \int_0^tnx^{n-1}F(x)dx
\sim\frac\alpha{n-\alpha}\,c\,t^{n-\alpha}.\]
Moreover, as $F$ is nonincreasing, we have a bound $F(x)\le c'x^{-\alpha}$ for
some constant~$c'$. This enables us to dominate~$M_{t,n}$ by:
\[M_{t,n}\le\frac{n}{n-\alpha}c't^{n-\alpha}.\]
Picking $\tau=t$, a dominated convergence argument and the expression
\eqref{dm} therefore show that the characteristic function $\psi_t(s)$
tends to the characteristic function~$\phi_\alpha(s)$ of the
Darling--Mandelbrot distribution. We conclude using Lévy's theorem.

If $\alpha = 1$, we have $M_{t,1}\sim c\log t$ as $t$ tends to infinity; if
$\alpha > 1$, $M_{t,1}$ tends to the finite value~$\mu$. This means that the
ratio~$M_{t,1}/[\tau F(t)]$ tends to~$1$ with the respective values $\tau =
t\log t$ and $\tau = t^\alpha\mu/c$. Moreover, in both cases, all the higher
moments satisfy $M_{t,n} = \bigo(t^{n-1})$ and are therefore negligible before
$\tau^nF(t)$. This means that $\psi_t(s)$ satisfies:
\[\psi_t(s)\to \frac1{1-is}\text,\]
which is the characteristic function of the exponential distribution.
\end{proof}

\section{The Darling--Mandelbrot density}
\label{sec.DMfacts}
\label{sec:density}

This section is devoted to the computation and the derivation of properties of
the density, denoted by~$g$, of the Darling--Mandelbrot distribution. By
studying the Laplace transform, Lew \cite[]{lew} determined that $g$ is a
continuous function satisfying:
\begin{alignat}{2} \label{gzero}
g(x) &= C_0x^{\alpha-1} + \bigo(1)\text,&\qquad x&\to0^+\text;\\ \label{ginf}
g(x) &= \frac{a_0}{\alpha}e^{-a_0(1+x)}+\bigo(e^{-a_1x})\text,&\qquad
x&\to\infty\text,
\end{alignat}
where $C_0$ is the constant:
\begin{equation} \label{C0}
C_0 = \frac{\sin(\alpha\pi)}{\pi} = \frac1{\Gamma(1-\alpha)\Gamma(\alpha)}
\end{equation}
and where $-a_0$ is the real zero of the function~$z\mapsto
z^\alpha\gamma(-\alpha,z)$ and $a_1 > a_0$ (see the reference for details).


\subsection{Explicit forms of the density}

\begin{theorem} \label{thm:density}
Let $0<\alpha<1$. The Darling--Mandelbrot density $g(x)$ is equal to:
\begin{equation} \label{g}
g(x) = \sum_{k=0}^\infty g_k(x)\text,
\end{equation}
where the function~$g_k(x)$ is continuous for $x > 0$, supported for~$x>k$,
analytic on its support, and has the two following equivalent definitions.
\begin{itemize}
\item Let $a(x)$ and $b(x)$ be the functions, supported for~$x>0$ and~$x>1$
respectively, defined by:
\begin{equation} \label{ab}
a(x) = C_0 x^{\alpha-1}
\qquad\text{and}\qquad
b(x) = -C_0 \frac{(x-1)^\alpha}{x}\text,
\end{equation}
where $C_0$ is defined by \eqref{C0}. Then $g_k(x)$ is equal to the
convolution product:
\begin{equation} 
\label{gk-convol}
g_k(x) = a*\underbrace{b*\dotsm*b}_{\text{$k$ times}}\,(x)\text.
\end{equation}
\item Let:
\begin{equation}
\beta_k = \alpha + k(1 + \alpha)\qquad\text{and}\qquad
C_k =\frac1{\Gamma(1-\alpha)\Gamma(-\alpha)^k\Gamma(\beta_k)}\text.
\end{equation}
Then $g_k(x)$ has the power series representation\footnote{The sum in this
expression can be seen as a special case of the Lauricella function
$F_B^{(k)}$ where all variables are specialized to~$-x$.} convergent for
$k<x<k+1$ and analytically continuable for $x\ge k+1$:
\begin{equation} \label{gk}
g_k(x) = C_k(x-k)^{\beta_k-1}\sum_{n_1,\dotsc,n_k\ge0}
\frac{(1{+}\alpha)_{n_1}\dotsm(1{+}\alpha)_{n_k}}{(\beta_k)_{n_1+\dotsb+n_k}}
(k-x)^{n_1+\dotsb+n_k}\text,
\end{equation}
where $(y)_n = \Gamma(y+n)/\Gamma(y)$ is the Pochhammer symbol for the rising
factorial.
\end{itemize}
\end{theorem}

\noindent
Again, we make some remarks before proving the theorem. First, since the
summand~$g_k(x)$ has support for~$x>k$, the infinite sum in \eqref{g} is
locally finite, which justifies its existence and shows that $g$ is
continuous. Moreover, since $b(x)$ is negative, the summands alternate in
sign. In particular, if $x\le1$, we have $g(x) = g_0(x) = C_0x^{\alpha-1}$.
This shows that the error term in \eqref{gzero} is in fact zero in that range.
In the case $k = 1$, the sum in \eqref{gk} takes the form of a hypergeometric
function:
\[g_1(x) = C_1(x-1)^{2\alpha}\,{}_2F_1(1,1+\alpha;1+2\alpha; 1-x)\text.\]
For $\alpha=1/2$, this simplifies into:
\[g_1(x) = \frac{x^{-1/2}-1}{\pi}\text.\]

The theorem also enables us to find the singularities of the density~$g(x)$.
Since the leading term in the sum in \eqref{gk} is~$1$, the function~$g_k(x)$
has a singularity at~$k$ of the form:
\[g_k(x) = C_k(x-k)^{\beta_k-1}\,\mathbf 1_{x>k}
+ \bigo\bigl((x-k)^{\beta_k}\bigr) \text.\]
Moreover, as the function~$g_k(x)$ is analytic for $x>k$ and the sum \eqref{g}
is locally finite, the density~$g(x)$ is singular at all integer points and
analytic otherwise, the singularity at the point $x = k$ being contributed
by~$g_k(x)$.

Finally, we note that although the sum in \eqref{g} behaves very well locally
(indeed, it's locally finite), it's not the case globally: as $x$ tends to
infinity, $g_k(x)$ behaves like $x^{k\alpha-1}$ and so alternately tends to
$\pm\infty$ for $k$ sufficiently large. Yet, as found by studying the Laplace
transform, the sum converges exponentially fast to zero.

In order to plot the density~$g(x)$, the most adequate characterisation is
\eqref{gk-convol}, or better yet, the differential equation of the forthcoming
Theorem~\ref{thm.equadiff}, which was used to produce Figure~\ref{fig:g}.

\begin{figure}[htb]
\setlength{\unitlength}{16pt}
\begin{picture}(21,5.7)(.2,0)
\put(0.1,.5){\includegraphics[scale=0.8]{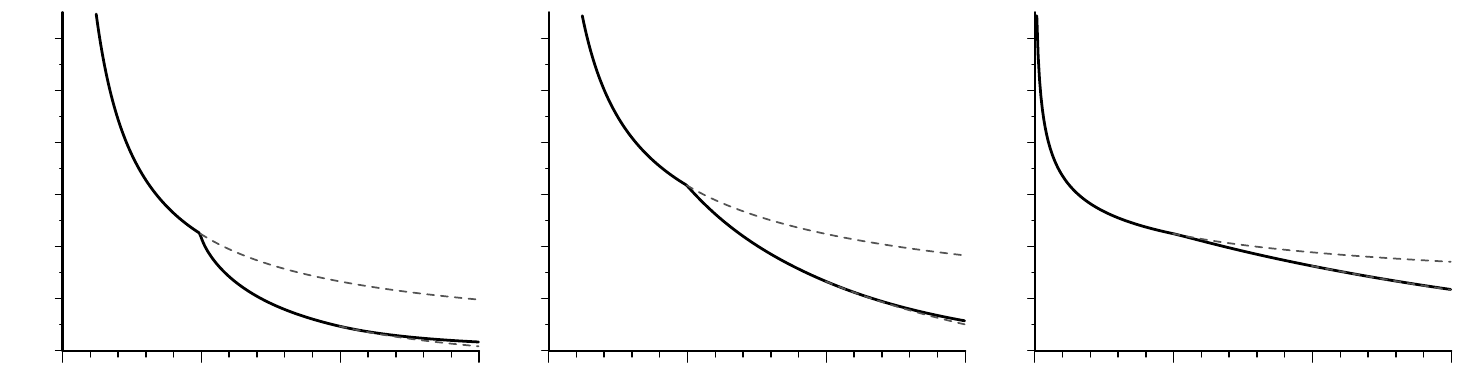}}
\put(1,0){\makebox[0pt][c]{0}}
\put(3,0){\makebox[0pt][c]{1}}
\put(5,0){\makebox[0pt][c]{2}}
\put(7,0){\makebox[0pt][c]{3}}
\put( 8,0){\makebox[0pt][c]{0}}
\put(10,0){\makebox[0pt][c]{1}}
\put(12,0){\makebox[0pt][c]{2}}
\put(14,0){\makebox[0pt][c]{3}}
\put(15,0){\makebox[0pt][c]{0}}
\put(17,0){\makebox[0pt][c]{1}}
\put(19,0){\makebox[0pt][c]{2}}
\put(21,0){\makebox[0pt][c]{3}}
\put(0.38,.55){\makebox[0pt][l]{0}}
\put(0.4,1.3){\makebox[0pt][l]{.1}}
\put(0.4,2.05){\makebox[0pt][l]{.2}}
\put(0.4,2.8){\makebox[0pt][l]{.3}}
\put(0.4,3.55){\makebox[0pt][l]{.4}}
\put(0.4,4.3){\makebox[0pt][l]{.5}}
\put(0.4,5.05){\makebox[0pt][l]{.6}}
\put(7.38,.55){\makebox[0pt][l]{0}}
\put(7.4,1.3){\makebox[0pt][l]{.1}}
\put(7.4,2.05){\makebox[0pt][l]{.2}}
\put(7.4,2.8){\makebox[0pt][l]{.3}}
\put(7.4,3.55){\makebox[0pt][l]{.4}}
\put(7.4,4.3){\makebox[0pt][l]{.5}}
\put(7.4,5.05){\makebox[0pt][l]{.6}}
\put(14.38,.55){\makebox[0pt][l]{0}}
\put(14.4,1.3){\makebox[0pt][l]{.1}}
\put(14.4,2.05){\makebox[0pt][l]{.2}}
\put(14.4,2.8){\makebox[0pt][l]{.3}}
\put(14.4,3.55){\makebox[0pt][l]{.4}}
\put(14.4,4.3){\makebox[0pt][l]{.5}}
\put(14.4,5.05){\makebox[0pt][l]{.6}}
\end{picture}
\caption{\label{fig:g}%
Plots of the density~$g(x)$ for $\alpha=1/4$, $\alpha=1/2$ and
$\alpha=3/4$ (from left to right). Dashed, the continuation of
the partial sums $g_0+\cdots+g_k$ beyond $x=k+1$.
The precision is far beyond line thickness (as easily obtained through
the characterisation of Theorem~\ref{thm.equadiff}).
} 
\end{figure}
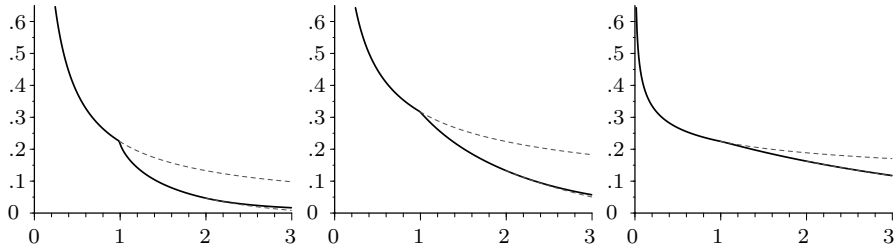

\begin{proof}[of Theorem~\ref{thm:density}]
Let us first prove the convolution product representation. From the identity
\eqref{dm}
we find the Laplace transform of~$g(x)$, that we denote
by~$G(z)$:
\[G(z) = \phi_\alpha(iz) = \frac{z^{-\alpha}}{-\alpha\gamma(-\alpha,z)}\text.\]
We transform this into:
\[G(z) =
\frac{z^{-\alpha}}{-\alpha\bigl(\Gamma(-\alpha)-\Gamma(-\alpha,z)\bigr)}
= \sum_{k\ge0}A(z)B(z)^k\text,\]
where $\Gamma(\cdot,\cdot)$ is the upper incomplete gamma function and:
\[
A(z) = \frac{z^{-\alpha}}{\Gamma(1-\alpha)}
\qquad\text{and}\qquad
B(z) = \frac{\Gamma(-\alpha,z)}{\Gamma(-\alpha)}
\]
(if $z$ is large enough so that $\abs{B(z)} < 1$; 
numerically, $\reof(z)>0.107878\ldots$\ suffices, uniformly for all $\alpha$).

Noting that $\Gamma(1-\alpha)\Gamma(\alpha) =
-\Gamma(-\alpha)\Gamma(1+\alpha)$, the following elementary computations show
that the functions $a(x)$ and $b(x)$ defined in \eqref{ab} have Laplace
transforms $A(z)$ and $B(z)$, respectively:
\begin{align*}
\int_0^\infty x^{\alpha-1} e^{-xz} dx
&= z^{-\alpha}\int_0^\infty x^{\alpha-1} e^{-x} dx\\
&= z^{-\alpha}\Gamma(\alpha)\text;\\
\int_1^\infty \frac{(x-1)^\alpha}{x} e^{-xz} dx
&= \int_z^\infty\int_1^\infty (x-1)^\alpha e^{-xy} dx dy\\
&= \int_z^\infty y^{-1-\alpha}\Gamma(1+\alpha)e^{-y}dy\\
&= \Gamma(-\alpha,z)\Gamma(1+\alpha)\text.
\end{align*}
Inverse Laplace transform thus yields \eqref{gk-convol}. The function $g_k(x)$
is analytic for $x>k$ as the convolution product of analytic functions.

Let us now prove the power series representation. Let $1<x<2$. A Taylor
expansion of the function $b(x)$ yields:
\[b(x) =
\frac1{\Gamma(-\alpha)}\sum_{n\ge0}(-1)^n\frac{(x-1)^{\alpha+n}}
{\Gamma(1+\alpha)}.\]
The identity \eqref{gk} then follows from \eqref{gk-convol} using the classic
formula:
\begin{equation} \label{convol}
f_{r_1}^{(\alpha_1)}*\dotsm*f_{r_k}^{(\alpha_k)} =
f_{r_1+\dotsb+r_k}^{(\alpha_1+\dotsb+\alpha_k)}\text,\qquad
f_{r}^{(\alpha)}(x) = \frac{(x-r)^{\alpha-1}}{\Gamma(\alpha)}\mathbf
1_{x>r}\text.
\end{equation}

Finally, since $g_k(x)$ is analytic for $x > k$, its value for $x\ge k+1$ is
found by analytic continuation.
\end{proof}


\subsection{Differential equations satisfied by the density}

In this section, we characterize the density~$g$ not explicitly, but
implicitly as the solution of differential equations. Since $g$ is singular at
all integer points, all differential equations are understood to be satisfied
only outside singular points.

\begin{theorem}
\label{thm.equadiff}
The density $g(x)$ is the only continuous solution of the non-linear
differential equation:
\begin{equation} 
\label{de}
xg'(x) + (1 - \alpha) g(x) = -\alpha\,g*g(x-1)\text,
\end{equation}
with initial condition \eqref{gzero}.
\end{theorem}

As the density~$g$ is positive, this result shows in particular that~$g$ is
decreasing. In fact, the equation above can be rewritten as
\begin{equation*}
\label{de_mod}
\frac{\dd}{\dd x}\big( x^{1-\a} g(x) \big)
=-\a x^\a
g*g(x-1)\text.
\end{equation*}
This makes evident the stronger statement that $x^{1-\alpha}g(x)$
is nonincreasing. This answers a question of Lew, who suggested that~$g(x)$
might show oscillations for small values of~$\alpha$.

\begin{proof}
Let us prove that $g$ satisfies the equation. One way to proceed is to
differentiate the Laplace transform $G(z)$. One can also directly use the
representations of Theorem~\ref{thm:density}.
Another way, that we detail here, is to compare the threshold sum processes at
thresholds $t$ and $u$, with $u\ge t$. We have:
\[Y_u = \begin{cases}
Y_t&\text{if $X_{I(t)} \ge u$;}\\
Y_t + X_{I(t)} + Y'_u&\text{if $t\le X_{I(t)} < u$,}
\end{cases}\]
where $Y'_u$ is independent from $Y_t$ and distributed like $Y_u$.

Now, set $u = \lambda t$ and let $t$ tend to infinity. The event $X_{I(t)} \ge
u$ occurs with probability $F(u)/F(t)\to \lambda^{-\alpha}$. If it does not,
we have $X_{I(t)} = t + \bigo(\lambda-1)$. Dividing by~$t$ and
recalling that $Y_t/t$ tends to the law of density~$g$, we get:
\[\lambda^{-1}g(\lambda^{-1}x) = \lambda^{-\alpha}g(x) +
(1 - \lambda^{-\alpha})\bigl(g*g(x-1) + \bigo(\lambda-1)\bigr)\text.\]
We recover \eqref{de} at first order in~$\lambda-1$.

To show the uniqueness of the solution, we note that the right hand side of
\eqref{de} depends only on the values of $g(y)$ for $y < x-1$; in particular,
it is zero for $x < 1$. This enables us to solve iteratively the equation on
the intervals~$[k,k+1]$, treating the equation as an inhomogenous linear
differential equation, with the initial value $f(k)$ found by continuity. This
determines the solution uniquely.
\end{proof}



Our final result writes the density~$g$ as the solution of linear differential
equations. Write $d_x = \dd/\dd x$ and let $D_k$ and $E_k$ be the
differential operators:
\begin{equation*}
D_k = d_x(x-k) - (k+1)\alpha\text;\qquad\qquad
E_k = D_{k-1}\dotsm D_0\text.
\end{equation*}

\begin{theorem} \label{thm:hol}
The operator $E_k$ cancels the functions~$g_0,\dotsc,g_{k-1}$ defined in
Theorem~\ref{thm:density}. In particular, it cancels the density~$g$ on the
interval~$[0,k]$.
\end{theorem}

\begin{proof}
To prove the theorem, we need the following elementary facts about convolution
products:
\begin{equation*}
x(u*v) = (xu)*v +u*(xv)\text;\qquad\qquad
(u*v)' = u'*v\text.
\end{equation*}
We first prove by induction that, for $0\le\ell\le k$, we have:
\[E_\ell\cdot g_k = \frac{k!}{(k-\ell)!} a_\ell*b^{*k-\ell}\text,\]
where
\[a_\ell(x) = \frac{(x-\ell)^{(\ell+1)\alpha-1}}
{\Gamma(1-\alpha)\Gamma(-\alpha)^\ell\Gamma\bigl((\ell+1)\alpha\bigr)}
\mathbf 1_{x>\ell}\text.\]
For $\ell=0$, this is obvious as $a_0 = a$. Otherwise, assume that the
identity is true at rank~$\ell$ and apply the operator~$D_\ell$ to it. Using
the above properties of convolution products, we have:
\[E_{\ell+1}\cdot g_k =
\frac{k!}{(k-\ell)!}(D_\ell\cdot a_\ell)*b^{*k-\ell}
+ \frac{k!}{(k-\ell-1)!} a_\ell*(xb)'*b^{*k-\ell-1}\text.\]
Since $D_\ell$ annihilates $a_\ell$, we conclude using the fact that
$a_\ell*(xb)' = a_{\ell+1}$ found using formula~\eqref{convol}.

At $\ell = k$, we thus find $E_k\cdot g_k = k!\,a_k$. Since $D_k\cdot a_k =
0$, we have indeed $E_\ell\cdot g_k = 0$ for $\ell > k$.
\end{proof}

\section{Applications}
\label{sec.appl}
\label{sec:applications}

\newcommand{\schrarrow}{\xrightarrow{\setlength{\unitlength}{3pt}\begin{picture}(4,1)\put(2,-1.5){\rule{.5pt}{3pt}}\end{picture}}}

In this section, we apply our results to the analysis in limit law of random
sampling algorithms. In all cases, this complexity is linked to a threshold
sum process that falls within the conditions of Theorem~\ref{thm:main}. Among
the three regimes in this theorem, the most favorable is the first one, with
the scaling factor~$t$ meaning that the algorithm has linear complexity.

In the following, we consider an anticipated rejection algorithm based on a
process with survival probability at time~$t$ asymptotic to~$c\,t^{-\alpha}$;
the algorithm consists in running the process repeatedly until it reaches
time~$t$. Since the successful run takes time~$t$, the complexity normalized
by~$t$ follows a Darling-Mandelbrot distribution \emph{shifted by one}, with
characteristic function $e^{is}\phi_\alpha(s)$ (this coincides with Darling's
initial definition). We denote by $\mathcal D(\alpha)$ this shifted
distribution.

In some cases, the algorithm has a second round of rejection on top of
anticipated rejection, \textit{i.e.}, it may fail and be restarted upon
reaching the target~$t$. Let us assume that it succeeds with a fixed
probability~$p$. The overall complexity of the algorithm is then of the form
$Y_1+\dotsb+Y_Z$, where the $Y_i$'s are independent variables following the
law $\mathcal D(\alpha)$ and~$Z\ge1$ is geometrically distributed with
parameter~$p$. Let $\mathcal D(\alpha,p)$ denote such a distribution
and~$e^{is}\phi_{\alpha,p}(s)$ be its characteristic function. We have:
\begin{equation} \label{dmp}
\phi_{\alpha,p}(s) =
\frac{p
\,\phi_\alpha(s)}{1 - (1-p)e^{is}\phi_\alpha(s)}
= 
\Biggl(1 - 
\sum_{n=1}^\infty \frac{\frac{1-p}{p} n + \alpha}{n-\alpha}\frac{(is)^n}{n!}
\Biggr)^{-1}\text.
\end{equation}
This situation typically arises when each step of the algorithm consists in
growing the sampled object by an increment $s_1,\dotsc,s_k$ with respective
probabilities $p_1,\dotsc,p_k$. In this case, there is a possibility that the
sample \emph{misses} the target size~$t$ by hopping over it. In the aperiodic
case (where $s_1\wedge\dotsb\wedge s_k = 1$), this occurs with an asymptotic
probability $p = 1/\delta$ where $\delta = \sum_i p_i s_i$ is the \emph{drift} of
the process. Slightly more subtle is the situation in which the
$s_i$'s are not all non-negative (but still the drift is positive), an
eventuality discussed in Section \ref{sec.holon}.
Examples are detailed below.

Let $Y$ be a random variable following the distribution~$\mathcal D(\alpha)$.
The moments of~$Y$ can be recovered by Taylor expansion of the expression
\eqref{dm} multiplied by~$e^{is}$. In particular, we have:
\[\mathbb E(Y) = \frac1{1-\alpha}\text;\qquad\qquad
\mathbb V(Y) = \frac\alpha{(1-\alpha)^2(2-\alpha)}\text.\]
As convergence in distribution implies convergence of moments, this will
enable us to compute the asymptotic behavior of the moments of the complexity
of the algorithms. The distribution~$\mathcal D(\alpha,p)$ can be treated in
the same way using \eqref{dmp}. This yields:
\[\mathbb E(Y) = \frac1{p(1-\alpha)}\text;\qquad\qquad
\mathbb V(Y) = \frac{\alpha +
2(1-p)(1-\alpha)}{p^2(1-\alpha)^2(2-\alpha)}\text.\]

\subsection{Prefixes of Motzkin paths and directed animals}

\newcommand\dm[1]{\smash{$\mathcal D\bigl(#1\bigr)$}}

The simplest algorithm that fits in our framework is probably the one
described in \cite[]{barcucci}, which samples prefixes of Motzkin paths (\textit{i.e.},
lattice paths with steps in $\{\nearrow,\searrow,\rightarrow\}$ never stepping
lower than their origin). Using a bijection of Penaud, they thus get a random
sampling algorithm for directed animals. A generalization appears in
\cite[]{barcucci}, which deals with the case where there are several possible
steps of each type (colored Motzkin prefixes).

The algorithm is very simple: the path is built by adding random steps one at
a time. If, at any time, the path steps below the origin, the algorithm is
started over from scratch. If the target size~$n$ is reached, the path is
output.
To our knowledge, this is the best known algorithm for exactly
sampling such structures, with the exception of the special case in
which there is no $\rightarrow$ step 
(\textit{i.e.}, \emph{prefixes of Dyck paths}).%
\footnote{To sample these, a better (in fact, optimal)
  algorithm consists in using the algorithm of \cite[]{motzkin} to
  sample a pointed binary plane tree and using classical bijections to
  get a Dyck prefix.}




\begin{proposition} \label{barcucci}
Assume the number of possible $\nearrow$ and $\searrow$ steps is the same.
Let $Y_n$ be the number of steps drawn by the algorithm to sample a path of
length~$n$. As $n$ tends to infinity, the random variable $Y_n/n$ tends in
distribution to the law~\dm{\frac12}.
\end{proposition}

In particular, we recover estimates of the expected value and variance given
by Barcucci et al., namely:
\[\mathbb E(Y_n)\sim 2n\text;\qquad\qquad\mathbb V(Y_n)\sim\frac43n^2.\]

\begin{proof}
Let $k$ be the total number of available steps. If there are as much different
$\nearrow$ and $\searrow$ steps, the number~$m_n$ of Motzkin prefixes of
length~$n$ satisfies $m_n\sim ck^nn^{-1/2}$, where $c$ is a constant.

Let $X$ be the random variable counting the number of steps before a random
path goes below the origin. We have $X > n$ if and only if the first $n$ steps
form a Motzkin prefix, which happens with probability $m_n/k^n\sim cn^{-1/2}$.

As outlined above, the random variable $Y_n$ onsists of two parts: the cost
of the unsuccessful trials, which follows a threshold sum process with base
distribution~$X$ and threshold~$n$, and the cost of the final successful
trial, which is~$n$. By Theorem~\ref{thm:main}, the quotient~$Y_n/n$ thus
tends to the shifted law~\dm{\frac12}.
\end{proof}

\subsection{Prefixes of Schröder paths}
\label{sec.schro}

A variant of the previous algorithm, sampling prefixes of Schröder paths, is
found in \cite[]{penaud}. A Schröder path has the same constraints as a
Motzkin path and takes steps in~$\{\nearrow,\searrow,\schrarrow\}$ (where
$\schrarrow$ has length~$2$). As shown in \cite[]{king}, these paths are also
in bijection with directed lattice animals, this time on the \emph{king's
lattice} (Figure~\ref{fig.dirpaths}).

The algorithm is similar to the one above, but the steps
$\nearrow,\searrow,\schrarrow$ are taken with respective probabilities
$\rho,\rho,\rho^2$ with $\rho = \sqrt2-1$. There is another difference: when
sampling for a target size~$n$, it is possible to jump from~$n-1$ to~$n+1$ by
generating a~$\schrarrow$. In this case, we must discard the path and
start over. As the following result shows, this modifies slightly the limit
behavior of the complexity while keeping it linear.

\begin{proposition} \label{schroeder}
Let $Y_n$ be the total length of the steps drawn by the algorithm to sample a
Schröder prefix of length~$n$. The random variable $Y_n/n$ tends in
distribution to the law~\dm{\frac12,\frac{2+\sqrt2}4}.
\end{proposition}

From \eqref{dmp}, we get the expected value and variance of~$Y_n$:
\[\mathbb E(Y_n)\sim\bigl(8-4\sqrt2\bigr)n\text;\qquad\qquad
\mathbb V(Y_n)\sim\frac{16}{3}\bigl(16-11\sqrt2\bigr)n^2\text.\]

\begin{proof}
Let $s_n$ be the number of Schröder prefixes of length~$n$ and $p_n$ be the
probability to reach one of them. As we have $s_n\sim c\rho^{-n}n^{-1/2}$, we
have $p_n\sim cn^{-1/2}$, where $c$ is a constant.

Let $X$ be the random variable counting the length of the path sampled before
it goes below the origin. The event $X\ge n$ can occur in two ways: either we
sample a Schröder prefix of length~$n$ or a prefix of length~$n-1$ followed by
a $\schrarrow$; the probability of this is
$p_n+\rho^2p_{n-1}\sim(1+\rho^2)cn^{-1/2}$. In the same way as for
Proposition~\ref{barcucci}, the time necessary to reach this tends in
distribution to~\dm{\frac12}.

Finally, out of the two above possibilities, we are interested only in the
case where we draw a Schröder prefix of length~$n$. This happens with
probability 
$p_n/(p_n+\rho^2p_{n-1})\to1/(1+\rho^2) = (2+\sqrt2)/4$.
The number of times the size~$n$ is reached is geometrically distributed,
hence the result.
\end{proof}

\begin{figure}[!tb]
\begin{center}
\vspace{2mm}
\includegraphics{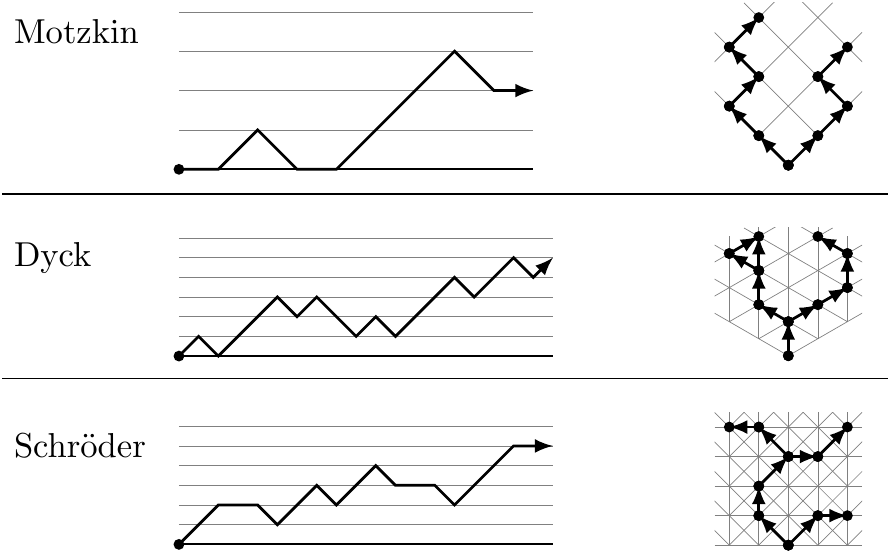}
\end{center}
\caption{\label{fig.dirpaths}Families of meanders in bijection with
  directed lattice animals.}
\end{figure}


\subsection{Unary-binary trees}
\label{sec.UniBin}

Another recent anticipated rejection algorithm appears in~\cite[]{motzkin},
sampling unary-binary plane trees. The algorithm works by letting a tree grow
from size~$1$ to~$n$ using a \emph{grafting} process akin to Rémy's algorithm
for binary trees, based on a holonomic equation. This process may fail,
however, in which case the algorithm is restarted. For our analysis, we use
the following two facts: first, the probability of reaching at least the
size~$n$ during the growth procedure satisfies $p_n\sim cn^{-1/2}$, with $c$ a
constant; second, at each step, the tree grows by~$1$ or~$2$ nodes with
respective probabilities $2/3$ and $1/3$. If this takes the size of the tree
from~$n-1$ to~$n+1$, the algorithm is restarted.

\begin{proposition}
Let $Y_n$ be the number of nodes of the trees built by the algorithm to sample
a tree with~$n$ nodes. The random variable $Y_n/n$ tends in distribution to
the law~\dm{\frac12,\frac34}.
\end{proposition}

Again, we deduce the expected value and variance from \eqref{dmp}:
\[\mathbb E(Y_n)\sim\frac83n\text;\qquad\qquad\mathbb V(Y_n)\sim\frac{32}9n^2
\text.\]

\begin{proof}
The proof is identical to the one of Proposition~\ref{schroeder}. The form
of the probability~$p_n$ shows that the time necessary to reach size~$n$ is,
normalized by~$n$, distributed like~\dm{\frac12}. Knowing we have reached at
least~$n$ nodes, the probability to hit exactly~$n$ is
$p_n/(p_n+\frac13 p_{n-1})\to \frac34$. This concludes the proof.
\end{proof}

\noindent
We remark that another way of sampling unary-binary trees with $n$
vertices is through the classical bijection with Motzkin excursions of
length $n-1$. These are in 1-to-$n$ correspondence with Motzkin paths
of length $n$ and ending at ordinate $-1$, which are themselves in
bijection with prefixes of Motzkin excursions, of length $n$ and
ending at odd ordinate. Such a prefix can be sampled using the
procedure of the previous section and a rejection scheme, but the
probability of rejection (checking if the final ordinate is odd) is
then asymptotically $1/2$ instead of $1/4$, leading to a slightly
worse complexity.


\subsection{More general holonomic systems}
\label{sec.holon}

The algorithmic strategy outlined in~\cite[]{motzkin} is potentially
ameanable to a variety of problems. Several combinatorial structures,
with a size parameter $n$, have generating functions $Z_n$ satisfying a
holonomic equation, \textit{i.e.}, an equation of the form
\be
\sum_{i \in I} P_i(n) Z_{n-i} = 0
\ee
where $I$ is a finite subset of $\mathbb{Z}$, and $P_i(n)$ are
polynomials with rational coefficients such that $P_0(n)\ne0$. Let $d$ the
maximal degree among the $P_i$'s, and $p_i$ the coefficient of degree $d$ in
$P_i$ (possibly zero). Asymptotically, we have
\be
\sum_{i \in I} p_i {Z_{n-i}}
= {Z_n}  \mathcal{O}(n^{-1})
\ef.
\label{eq.386243}
\ee
Suppose that the holonomic equation can be rewritten as
\be
\binom{n}{d}
Z_n = \sum_{i \in I} P_i(n) Z_{n-i}
\label{eq.48765654343}
\ee
(up to a redefinition of $P_0$), so that the coefficients of the
$P_i$'s are \emph{positive} rationals, when $P_i$ is written in the
polynomial basis $\binom{n-i}{k}$.\footnote{The condition on the form
  of the left-hand side can be relaxed to some extent, we treat here a
  simplified situation in order to lighten the notation.}  We can
interpret the $k$-th basis polynomial as associated to the enumeration
of objects with $k$ marked unit elements. The positivity of the
coefficients may prelude to the design of a bijective interpretation
of relation (\ref{eq.48765654343}), in which the marks undergo a local
dynamics, implemented with small complexity. We have an analogue of
equation (\ref{eq.386243}), of the form
\be
{Z_n} - \sum_{i \in I} p_i {Z_{n-i}}
= {Z_n}
\mathcal{O}(n^{-1})
\ef,
\label{eq.386243b}
\ee
where now $p_i$ is the coefficient of $\binom{n-i}{d}$ in $P_i$. Let
$\lim_{n \to \infty} \ln Z_n/n = \zeta$.  Define the \emph{drift}
$\delta$ as the average of $i \in I$, according to the distribution
$p_i \zeta^{-i}$ (which is normalised).  The bijection discussed above
can be turned into an algorithm, possibly of anticipated rejection.
This is what happens in~\cite[]{motzkin}, for binary and unary-binary
trees. In the first case the algorithm has no reject, in the second case 
anticipated rejection is required. Anticipated rejection may be needed
when the bijection involves, on the RHS, a combinatorial object with
less than $d$ marks. In some cases, the missing marks can be resampled
uniformly without introducing any bias, while in other cases this is
not possible, and the growing object has to be rejected.

The size at each algorithmic step $t$ changes by a random value $i_t
\in I$. This happens asymptotically with probabilities $p_i
\zeta^{-i}$. If the drift is positive, the size makes a directed
random walk with a positive slope, which, with high probability,
either intersects $n$ after $\sim n/\delta + \mathcal{O}(\sqrt{n})$
steps, or hops over this value and goes towards infinity. Thus, if
anticipated rejection is required, with exponent $\a$ in the
appropriate range, we are in the context of the geometric convolution
of the Darling--Mandelbrot distribution discussed at the beginning of
the section. 

If $I \subset \mathbb{N}^+$ (we say that $I$ is
\emph{non-backtracking} in this case), the walk either passes by $n$
exactly once, or misses it; asymptotically, this happens with
probability $1/\delta$ and $1 - 1/\delta$, respectively (provided $I$
is aperiodic, that is, has no common divisor $>1$). If the value $n$
is missed, we shall restart the algorithm.

If $I$ has support on both positive and negative integers (and thus is
\emph{backtracking}), the walk may intersect $n$ more than once, and
the first hit of $n$ may occur after that larger values have been
reached. This makes the optimisation and analysis of the algorithm
slightly more complicated.  Any of the hitting events gives an
unbiased sample, and a concrete algorithm will just take the first
one. Having a positive number of hitting events happens now with
probability smaller than $1/\delta$, but still $\mathcal{O}(1)$ (the
exact asymptotic probability involves a complicated expression in the
$p_i$'s, an analysis postponed to the following paragraphs). At any
time, possibly in light of the current size parameter, we have the
right of restarting the algorithm.  Restarting as soon as a value
higher than $n$ is attained is a feasible choice, but non-optimal by a
constant factor in complexity, as at values $n'=n+\mathcal{O}(1)$ we
still have a probability $\mathcal{O}(1)$ of hitting $n$ in
$\mathcal{O}(1)$ further steps, that largerly pays off against the
expensive restart procedure.  It is more efficient to restart the
algorithm as soon as we can confidently suppose that, with high
probability, the walk has reached a size larger than $n$ for never
coming back.  Based on the universal behaviour of drifted
one-dimensional random walks, a generic simple such strategy, which is
asymptotically optimal, is to restart as soon as the current size
reaches $n+\sqrt{n}$.

We now analyse the probability of having a positive number of
intersections, in the backtracking case.  The asymptotic probability
$\pi_s$ that there are $s$ intersections has the form $\pi_0 = a$, and
$\pi_s = b c^{s-1}$ for $s \geq 1$. This is due to the fact that
bridges at height $n$ are independent events, and are thus
concatenated geometrically. The resulting convoluted distribution is
thus $\mathcal{D}(\a,a)$, and we shall determine~$a$.

Normalisation gives $a+b(1-c)^{-1}=1$. The average number of
intersections is $b(1-c)^{-2}$, and must be given by
$\delta^{-1}$. However, these two trivial equations are not sufficient
to determine $a$, and we need a third relation. We can determine the
average of $\binom{s}{2}$, divided by the average of $s$, which is
$c/(1-c)$. In fact, the latter is represented combinatorially by a
path crossing $n$, in which one of the crossings is marked, and the
former is the analogous event in which two crossings are marked. The
two semi-infinite parts of the walk have analogous distributions in
the two processes, and the part of the walk between the two marks is a
random bridge, independent from the rest of the path. So, this
accounts to evaluate the generating function of bridges, at
criticality, for the asymptotic step rates $p_i \zeta^{-i}$. Such a
quantity is written as a Cauchy integral involving the kernel
(Laurent) polynomial, $K(\omega) = \sum_i p_i \omega^i$, and is
written in terms of the residues at those roots of the polynomial,
which are series in the parameter $\omega$ with no Laurent part
(called \emph{small roots}, see~\cite[]{BanFlaj}).

It is legitimate to ask whether there exist concrete applications in
which the set $I$ described above is backtracking.%
\footnote{This question has been posed also by the anonymous referee.}
A detailed discussion of this point would be besides the scope of the
present article. Let us however provide a simplistic example, of a
recursion in which $Z_n$ is a rational series, satisfying a holonomic
equation with \emph{constant} coefficients. The example shall
illustrate how backtracking recursions may arise easily from small
modifications of non-backtracking problems, preserving the
probabilistic interpretation of the associated generating series.  It
is well known that Fibonacci numbers satisfy the recursion $F_n =
F_{n-1} + F_{n-2}$, with suitable initial conditions $F_0=0$ and
$F_1=1$. Such a recursion has set $I=\{1,2\}$, thus it does not
provide an example of the form we seek.
%
These numbers can be refined to integer-valued polynomials, e.g.\ as
$F_n' = F_{n-1}' + x F_{n-2}'$, or as $F_n'' = x F_{n-1}'' + x
F_{n-2}''$ (in the two cases, again for suitably chosen initial
conditions, the polynomials are trivially related: $F_n''(x)=x^n
F_n'(x^{-1})$). In our prespective of exact sampling, $n$ is the size,
and $x \in \mathbb{R}^+$ is a parameter in the measure on the
associated combinatorial objects (Fibonacci trees, or dimer-monomer
configurations on an interval).  Combining the equations at two
consecutive sizes, we have $(1+x) F_n' = F_{n+1}' + x^2 F_{n-2}'$, and
$(1+x) F_n'' = F_{n+1}'' + x F_{n-2}''$, in the two cases.
For $x \in \mathbb{R}^+$, both these equations have set $I=\{-1,2\}$,
and are thus backtracking.
Of course, we cannot be satisfied with these examples either: these
quantities satisfy also the simpler customary Fibonacci relations,
which provide a simpler, non-backtracking implementation of the
sampling algorithm. In other words, both of the associated
polynomials, $1-w(1+x)+w^3 x=(1-w)(1-wx-w^2x)$ and $1-w(1+x)+w^3
x^2=(1-wx)(1-w-w^2x)$, factorise.

Consider now any convex combination of the two relations, e.g.
\be 
2(1+x) F_n = 2 F_{n+1} + x(x+1) F_{n-2} 
\ee 
which has associated polynomial $1-w(1+x)+\frac{1}{2} w^3 x(x+1)$, that is not
factorisable.  Still, under suitable initial conditions (such as
$F_0=0$, $F_1=1$, $F_2=1+x$), the polynomials $2^n F_n(x)$ have
positive integer coefficients, with a potential combinatorial
interpretation.\footnote{The positivity property still holds
for the homogeneous, more refined polynomials associated to the
equation
\begin{gather*}
(2x+y+z) 
F_n = F_{n+1} + 
x(y+z)(2x+y+z) F_{n-2}
\\
F_0=0\,;\quad F_1=1\,;\quad F_2=x + y\,.
\end{gather*}
}


\subsection{Random walks in conical domains}

In this section, we study models of constrained random walks. The
complexity of the anticipated rejection algorithm is governed by the
survival probability of the model, that is, the probability of a
random walk of length~$t$ to satisfy the constraints. The analysis of
survival probability for this class of problems has a long history,
that dates back at least to Sommerfeld at the beginning of the
century. A review of results can be found in \cite[Chapter~7]{redner},
and a modern approach with a rigorous derivation can be found in
\cite[]{Denis}.\footnote{We thank M.~Bousquet-M\'elou and K.~Rashel
  for pointing out this reference.}

The first case we describe is random walks in the square lattice
constrained to remain in a wedge of angle~$\theta$. As explained in
\cite[Section~7.2]{redner}, the survival probability satisfies in this
case $F(t)\sim ct^{-\alpha}$ where $2\theta \alpha = \pi$.  An
identical result holds for other regular lattices (such as triangular,
hexagonal,\ldots).\footnote{This is not discussed in the synthetic
  presentation of \cite[]{redner}, but it could be derived on
  identical basis, and it is implicit in the large generality of the
  results in \cite[]{Denis}.}
This gives an exponent $\alpha$ ranging from $1/4$ (for excluding just
a half-line) to arbitrarily large (for a narrow wedge); however,
arbitrarily small values of~$\alpha$ can be found by considering
values of $\theta$ greater than~$2\pi$, by taking into account the
winding number of the walk. In particular, we find:
\begin{itemize}
\item for $\theta > \pi/2$, the algorithm has linear average complexity and
limit law \dm{\frac\pi{2\theta}}.
\item for $\theta = \pi/2$, the algorithm has average complexity $n\log n$
and exponential limit law;
\item for $\theta < \pi/2$, the algorithm has average complexity
\smash{$\bigo\bigl(n^{\frac\pi{2\theta}}\bigr)$} and an exponential limit law.
\end{itemize}
For specific values of~$\theta$, these walks can be realized as
\emph{walks in the quarter plane} with some prescribed steps
\cite[]{bousquet}. For instance, Gessel's walks, with steps
$\{\swarrow,\leftarrow,\nearrow,\rightarrow\}$, correspond to walks in
the square lattice in a wedge of angle~$3\pi/4$. The variant of
Kreweras' walks with steps
$\{\downarrow,\swarrow,\leftarrow,\uparrow,\nearrow,\rightarrow\}$
correspond to walks in the triangular lattice in a wedge of
angle~$2\pi/3$. Similarly, the other two variants with steps
$\{\downarrow,\leftarrow,\nearrow\}$ and
$\{\swarrow,\uparrow,\rightarrow\}$ correspond to walks in the same
wedge, for the oriented triangular lattice, in which the edges are
oriented in an alternating way around each vertex, the two families of
walks corresponding to the two possible orientations (see Figure~\ref{fig.krewges}).
The anticipated rejection algorithm thus has linear
complexity in both cases, with respective limit laws \dm{\frac23} (Gessel) and
\dm{\frac34} (Kreweras, in the three realisations).

\begin{figure}[!tb]
\begin{center}
\vspace{2mm}
\includegraphics{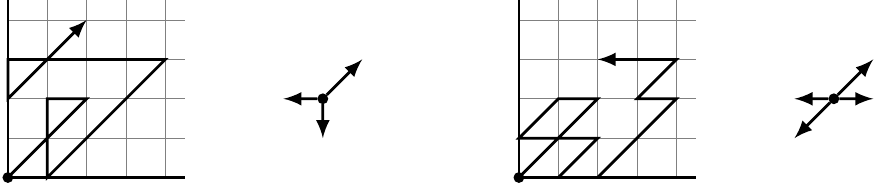}
\end{center}
\caption{\label{fig.krewges}Left: example of Kreweras' walk. Right:
  example of Gessel's walk.}
\end{figure}

A more complicated case is random walks in $\integers^3$ constrained in a cone
defined by $\theta < \theta_{\max}$ (in spherical coordinates). In this case,
the survival probability is $F(t)\sim ct^{-\nu/2}$, where $\nu$ is the
smallest positive number such that $P_\nu(\cos\theta_{\max}) = 0$, where
$P_\nu$ is the Legendre function \cite[Section~7.3]{redner}. This allows for
the exponent $\alpha = \nu/2$ to be any positive number. In particular, since
$P_2(x)  = (3x^2-1)/2$, an exponent $\alpha < 1$, and thus a linear-time
algorithm, is achieved for $\theta_{\max} > \arccos(1/\sqrt3)$.

In fact, generic cones in generic dimensions, and for a large class of
periodic lattices, can also be handled in this way. Full details can
be found in
\cite[]{Denis}, and, in particular, their Section 1.2 illustrates the
required precise technical assumptions. Let us summarise in few words
these hypotheses. There are four of them.  The first two are mild
requests on the shape of the cone, which in particular are
automatically satisfied in dimension 2. A third hypothesis allows for
long-range walk steps, provided that certain moments are finite (we
only considered walks with finite-range steps here). A fourth
hypothesis requires that the associated unbounded random walks undergo
isotropic diffusion, and always holds in absence of drift (as we
require here for having non-trivial asymptotics), up to applying an
appropriate affine transformation to the lattice.

Let $\Omega$ be a cone of $\reals^d$ and let 
$\Omega_0 = \Omega\cap\mathbb S^{d-1}$. Under the conditions detailed
in the reference, the survival probability in the cone~$\Omega$
satisfies:
\[F(t)\sim ct^{-\nu/2}\text,\]
where $\nu$ is the smallest positive number such that there exists a function
$h_\nu(\theta)$ on the unit sphere vanishing at the border of $\Omega_0$ and
satisfying:
\[\Delta_S h_\nu(\theta) = -\lambda h_\nu(\theta)\text,\qquad\qquad
\lambda = \nu(\nu + d - 2)\text,\]
where $\Delta_S$ is the spherical Laplace operator.

Thus, we are again in the conditions of Theorem~\ref{thm:main}, with $\alpha =
\nu/2$. The exponent $\nu$ is, however, difficult to compute in general.

\subsection{More complex random walk problems}

In the previous section we considered random walks on a lattice that
shall avoid some ``wall'' prescribed deterministically. Here we
consider a more complex problem in which the growing structure
produces the walls dynamically.  Say that two paths on a graph
\emph{intersect} if they share some vertex.  We have two classes of
problems: (P1) a walk $\omega$ of length $n$, starting at a neighbour
of the origin, such that there exists some infinite walk connecting
the origin to infinity and not intersecting $\omega$.  (P2) for $k
\geq 2$, $k$-tuples of walks $(\omega_1, \ldots, \omega_k)$, of length
$n$, starting from nearby vertices (e.g., aligned along a line), that
shall not intersect each other.


For undirected random walks, the simplest lattice is $\mathbb{Z}^D$,
\textit{i.e.}, with the $2D$ possible steps $\{ s^{\a} \}=\{ (0,0,\ldots,\pm
1,\ldots,0) \}$ uniformly chosen.  The analogue for directed random walks is
$\mathbb{N} \times \mathbb{Z}^{D-1}$, \textit{i.e.}, with the $2(D-1)$ possible steps
$\{ s^{\a} \}=\{ (1,0,\ldots,\pm 1,\ldots,0) \}$ uniformly chosen.
More generally, we may consider unbiased isotropic
(undirected) random walks,
\textit{i.e.}, walks that can perform steps $s^{\a} \in \mathbb{Z}^D$
with weight $w_{\a}$, such that $\sum_\a w_{\a} s_i^{\a} = 0$ for all
$i=1,\ldots, D$ and $\sum_\a w_{\a} s_i^{\a} s_j^{\a} = C \delta_{ij}$. In
the
directed variant we have $s_1^{\a} = 1$ for all steps $\a$, and all
other compatible constraints are left as are.


The associated exponents, when non-trivial 
(\textit{i.e.}, for $D$ sufficiently small),
are in general hard to evaluate. 
For directed walks in $D=2$, (P1) is
trivial, and (P2) is called \emph{vicious walkers}. The well-known relation
with classical ensembles of random matrices gives $\a=k(k-1)/4$ in that case.
This means that we have no problems in the interesting range $0<\a<1$, except
for $k=2$, which, on $\mathbb{N} \times \mathbb{Z}$,
reduces to prefixes of Dyck
paths through a simple bijection%
\footnote{It is worth noting that,
still on $\mathbb{N} \times \mathbb{Z}$, and at generic $k$,
in the variant in which the endpoints are prescribed,
exact enumeration formulas allow for an efficient
algorithm, involving no anticipated rejection (see
\cite[Chapt.\;4]{Bonich}).
We thank the anonymous referee for pointing us
  towards this reference.}.

For undirected walks in $D=2$, conformal invariance, and even better
the connection with the exactly solvable analysis on random planar
graphs via KPZ relation (\cite{kpz}), have led to the determination of
a variety of critical exponents, which have been proven subsequently
by SLE techniques (see \cite{Dupla} for the original conjectures, and
\cite{LSW1, LSW2, LSW3} for the proofs).  As shown in \cite{LSW2}, we
have $\a=\frac{1}{24}(4k^2-1)$, in a unified formula for (P1) (using
$k=1$) and for (P2).\footnote{Incidentally, note that also in the
  directed case the formula for $\a(k)$ matches with the trivial value
  $\a=0$ for problem (P1).}
Thus we have two new problems in the interesting range: problem (P1),
following the law $\mathcal{D}(\frac{1}{8})$, and problem (P2) with
$k=2$, following the law $\mathcal{D}(\frac{5}{8})$. An example of the
latter is in Figure~\ref{fig.dupl}.

\begin{figure}[!tb]
\begin{center}
\vspace{2mm}
\includegraphics{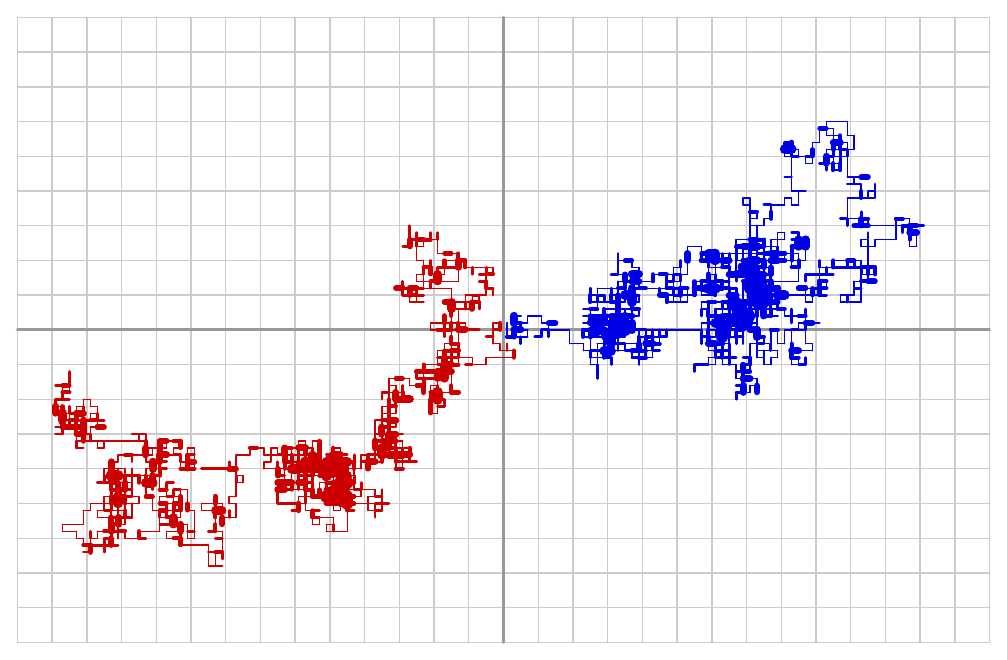}
\end{center}
\caption{\label{fig.dupl}A random sampling of two walks, each of
  length $n=2000$, on $\mathbb{Z}^2$, starting at neighbouring points
  and avoiding each other.}
\end{figure}

\bibliographystyle{abbrvnat}
\bibliography{biblio_short}



\end{document}